\documentclass[12pt]{amsart}
\usepackage{latexsym}
\usepackage{amsthm}
\usepackage{amsmath}
\usepackage{amssymb}
\usepackage{verbatim}
\usepackage{color}
\newtheorem{theorem}{Theorem}
\newtheorem{prop}[theorem]{Proposition}
\newtheorem{corollary}[theorem]{Corollary} 
\newtheorem{lemma}[theorem]{Lemma}

\def \n{\noindent }
\def \bs{\bigskip}

\def \R{\mathbb R}
\def \Q{\mathbb Q}

\def \Zc{\mathcal Z}

\def \Z{\mathbb Z}
\def \I{\mathcal I}
\def \F{\mathcal F}
\def \H{\mathcal H}
\def \B{\mathcal B}
\def \S{\mathcal S}
\def \bo{{\bf 1}}

\title{Level Matrices}
\author{G.~Seelinger}
\author{P.~Sissokho}
\author{L.~Spence}
\author{C.~Vanden~Eynden}
\address{Mathematics Department, Illinois State University, Campus Box
   4520, Stevenson Hall 313, Normal, IL 61790-4520}
\email{\{gfseeli|psissok|spence|cve\}@ilstu.edu}
\begin{document}
\begin{abstract}
Let $n>1$ and $k>0$ be fixed integers. A matrix is said to be {\em level} if all its column sums are equal. 
A level matrix with $m$ rows is called {\em reducible} if we can delete $j$ rows, $0<j<m$, so that 
the remaining matrix is level. 
We ask if there is a minimum integer $\ell=\ell(n,k)$ such that for all $m>\ell$,  
any $m\times n$ level matrix with entries in $\{0,\ldots,k\}$ is reducible.
It is known that $\ell(2,k)=2k-1$. In this paper, we establish the existence 
of $\ell(n,k)$ for $n\geq 3$ by giving upper and lower bounds for it. We then apply this 
result to bound the number of certain types of vector space multipartitions.
\end{abstract}
\maketitle
\section{Introduction}
Let $n>1$ and $k>0$ be integers. We define a $k$-{\em matrix} to be a matrix whose entries are in 
$\{0, 1, 2, \ldots, k\}$. 
A matrix is said to be {\em level} if all its column sums are equal. 
A level matrix with $m$ rows is called {\em reducible} if we can delete $j$ rows, $0<j<m$, so that 
the remaining matrix is level; otherwise it is {\em irreducible}. Note that if $M$ is an irreducible
matrix, then any matrix obtained from it by a permutation of rows or columns is also irreducible.  

For $k=1$ and any integer $n>1$, the $n\times n$ identity matrix is irreducible.
If $n>4$, then we can construct an irreducible $1$-matrix with $n$ columns and $m>n$ 
distinct rows. Moreover, for any integers $k>1$ and $n>1$, we can construct an irreducible $k$-matrix 
with $n$ columns and $m>n$ distinct rows. 

In general, we do not require that irreducible $k$-matrices have distinct rows.
We are interested in the following question.

\vspace{0.2cm}
\n{\bf Question}.
{\em Given integers $n>1$ and $k>0$, is there a minimum integer $\ell=\ell(n,k)$ such that for all $m>\ell$,  
any $m\times n$ level $k$-matrix is reducible?
If $\ell(m,k)$ exists, then what can we say about its value?
}

\vspace{0.2cm}
 The exact value of $\ell(2,k)$ (see Theorem~\ref{thm:La}) follows from earlier work by Lambert~\cite{Lam}. 
Perhaps due to a wide range of notation and terminology in related areas, Lambert's result has been 
(independently) rediscovered by Diaconis et al.~\cite{DGS}, and Sahs et al.~\cite{SST}.
In addition, M. Henk and R. Weismantel~\cite{HW2} gave improvements of Lambert's result.
\begin{theorem}[Lambert~\cite{Lam}]\label{thm:La}
If $k>1$, then $\ell(2,k)=2k-1$. Moreover, there are (up to row/column permutations) only two 
irreducible $k$-matrices with $2k-1$ rows and $2$ columns.
\end{theorem}
 However, to the best of our knowledge, the exact value of $\ell(n,k)$ is unknown for $n\geq 3$. 
In this paper, we prove the following theorem.
\begin{theorem}\label{thm:main} Let $n\geq 3$ and $k>0$ be integers, and  let $\epsilon>0$ 
be any real number. There exist infinitely infinitely many values of $n$ for which 
$\ell(n,1)>e^{(1-\epsilon)\sqrt{n\ln n}}$. On the other hand, 
\[ \ell(n,k)\leq 
\begin{cases} (2k)^3 \quad & \mbox{if $n=3$,} \\ 
k^{n-1}2^{-n}(n+1)^{(n+1)/2}\left((k+1)^n-k^n+1\right) 
& \mbox{if $n>3$.}\end{cases} 
\]
\end{theorem}
%
Let  $M_{m,n}(\Z)$ be the set of all $m\times n$ matrices with entries in $\Z$.
In what follows, vectors are assumed to be column vectors (unless otherwise specified), and $\Z^n$
denotes the set of all (column) vectors with $n$ entries from $\Z$.
Let $A\in M_{m,n}(\Z)$ be a $k$-matrix and let $\bo =(1,\ldots,1)^T \in \Z^n$.  We also let 
$A_i$ denote the $i$th row of $A$.

For any $\vec{x} \in \Z^m$, we say $\vec{x}$ is a {\em leveler of $A$} if $\vec{x}$ has nonnegative entries
and  there exists a nonnegative $\alpha \in \Z$ such that $\vec{x}^TA = \alpha \bo^T$, or equivalently, 
$A^T\vec{x} = \alpha \bo$.
In particular, given a leveler $\vec{x} =(x_1, x_2, \ldots,x_m)^T\in \Z^m$ of $A$, we can form
an $(x_1+x_2+\cdots+x_m)\times n$ level $k$-matrix by taking $x_1$ copies of the first row of $A$, 
$x_2$ rows of the second row of $A$, etc.  For the purposes of level $k$-matrices, this process will define 
this matrix up to a permutation of rows.  In this way, each leveler represents a class of level $k$-matrices.

One way to classify level $k$-matrices then is to classify the levelers of the $k$-matrices $A$.  
To assist us in this analysis, we use the base field ${\mathbb{Q}}$ and the following notation.
For any $\vec{x}\in{\mathbb{Q}}^m$, we write $\vec{x}\geq \vec{0}$ if $x_i\geq 0$ for all $1\leq i\leq m$.
If $\vec{x},\vec{y}\in {\mathbb{Q}}^m$, we write $\vec{x}\geq \vec{y}$ if $\vec{x}-\vec{y}\geq \vec{0}$, 
and write $\vec{x}>\vec{y}$ if $\vec{x}\geq \vec{y}$ and $\vec{x}\not=\vec{y}$.
Finally, let $\vec{x}\in \Z^m$ be a leveler for $A$.  We say $\vec{x}$ is an {\em irreducible} leveler if, 
for any leveler $\vec{y}\in \Z^m$, we have 
$\vec{x}\geq \vec{y} \Rightarrow \vec{y}=\vec{x} \mbox{ or } \vec{y} = \vec{0}$. 
Note that $\vec{x}$ is an irreducible leveler of $A$ if and only if the corresponding matrix formed 
from $A$ is an irreducible $k$-matrix.

Assume that the rows of $A\in M_{m,n}(\Z)$ are distinct and $m\geq n$. Define
\begin{equation}\label{FA}
\F(A)=\{\vec{x}\in {\mathbb{Q}}^m \,|\, A^T\vec{x}=\bo \mbox{ and } \vec{x}>\vec{0}\}.
\end{equation}
Note that $\F(A)$ is a convex polytope in ${\mathbb{Q}}^m$ since it is the intersection of the 
linear space $\{\vec{x}\in {\mathbb{Q}}^m \;|\; A^T\vec{x}=\bo\}$ with the half-spaces
$\H_i = \{ \vec{x}\in {\mathbb{Q}}^m \;|\; x_i \geq 0\}$ for $1\leq i\leq m$.

We say that $\vec{x}\in\F(A)$ is a {\em basic feasible solution} (BSF) 
if there exists a set of $n$ indices $I=\{i_1,\ldots,i_n\}\subseteq \{1,2,\ldots,m\}$ such that:
\begin{enumerate}
\item[(a)] $x_i=0$ for each $i\not\in I$.
\item[(b)] If $C$ is the matrix with rows $A_i$ for $i\in I$, then $C$ 
is invertible. Thus, if $\vec{y}=C^{-1}\bo$, then $x_{i_j}=y_j$ for $1\leq j\leq n$.
\end{enumerate}
Note that for any given set of $n$ indices $I$, there is at most one BFS corresponding to it.  
We will use this property later.

Define
\begin{equation}\label{BA}
\B(A)=\{\vec{x}\in \F(A)\,|\, \mbox{$\vec{x}$ is a basic feasible solution in $\F(A)$}\}.
\end{equation}
Let ${\mathcal C}(A) = \{ q\vec{x} :\;  q\geq 0, q\in {\mathbb{Q}}, \vec{x}\in \F(A)\}$ 
be the positive affine cone of $\F(A)$ in ${\mathbb{Q}}^m$. Then ${\mathcal C}(A)$ is a pointed 
rational cone generated by $\B(A)$, and ${\mathcal Z}(A)={\mathcal C}(A)\cap \Z^m$ is exactly 
the set of levelers for $A$. By \cite[Prop.\ 7.15]{MS}, there exists a unique minimal generating 
set of ${\mathcal Z}(A)$, which is called the {\em Hilbert basis} of ${\mathcal Z}(A)$.  
We have the following proposition.
\begin{prop}\label{prop:hil} If $A$ is a matrix with nonnegative 
entries, then the Hilbert basis of ${\mathcal Z}(A)$ is the set of irreducible levelers of $A$. 
\end{prop}

In Section~\ref{sec:main}, we prove our main theorem (Theorem~\ref{thm:main}) using tools 
from combinatorial optimization (in particular Carath\'{e}odory's Theorem).
In Section~\ref{sec:app-vsp}, we apply Theorem~\ref{thm:main} to prove some Ramsey-type 
statements about vector space multipartitions with respect to some 
irreducibility criteria that we shall define later.
\section{Proof of Theorem~\ref{thm:main}}\label{sec:main}
The proof of our main theorem relies on Theorem~\ref{thm:UB1} in Section~\ref{sec:UB1}, 
Theorem~\ref{thm:UB2} in Section~\ref{sec:UB2}, and Theorem~\ref{thm:LB} in
Section~\ref{sec:LB}.
\subsection{The first upper bound for $\ell(n,k)$}\label{sec:UB1}\

For any matrix $A$, let $\ell(A)$ denote its number of rows, 
and let $|A|$ denote its determinant if $A$ is a square matrix.
For any rational vector $\vec{x}$, let $\vec{x}_i$ denote its $i$th entry and let 
$r_x$ be the smallest positive integer such that $r_x\vec{x}$ is integral, i.e., the entries 
of $r_x\vec{x}$ are all integers. 

For any vector $\vec{x}\in \F(A)$ (thus, $A^T\vec{x}=\bo$ and $\vec{x}>\vec{0}$) and any 
integer $r> 0$ such that $r\vec{x}$ is an integral vector, let $L(A,r,\vec{x})$ be the matrix obtained by 
stacking $rx_i$ copies of $A_i$ for $1\leq i\leq m$.  Note that we define $L(A,r,\vec{x})$ 
up to a permutation of rows.

\begin{lemma}\label{bfs:irr}
 If  $A\in M_{n,n}(\Z)$ is an invertible $k$-matrix and  $\vec{x}\in\F(A)$, then  $L(A, r_x,\vec{x})$  
is irreducible.
\end{lemma}
\begin{proof}
If the lemma does not hold, then there exists  $\vec{y} \in \Z^n$,  with $\vec{0}<\vec{y}<r_x\vec{x}$, 
such that  $A^T \vec{y}$ is level, say with column sums $t>0$. Then  $A^T\vec{y}=t\bo$, so  
$A^T t^{-1}\vec{y}=\bo=A^T\vec{x}$.
Thus  $t^{-1}\vec{y}=\vec{x}$.  But then  $t\vec{x}=\vec{y}$  is integral, so by the definition of  
$r_x$  we have $t \geq r_x$. This contradicts the assumption that  $\vec{y}<r_x\vec{x}$.
\end{proof}
A {\em convex combination} of the vectors $\vec{x}^{(1)},\ldots, \vec{x}^{(t)}$ is an expression of the form
\[\lambda_1\vec{x}^{(1)}+\ldots+\lambda_t\vec{x}^{(t)}\mbox{ with } \lambda_i\in \R^+
\mbox{ for $1\leq i\leq t$, and }\sum_{i=1}^t\lambda_i=1.\]
If $\lambda_i\in\Q^+$ for all $1\leq i\leq t$, then the convex combination is called {\em rational}. 
%
\begin{lemma}\label{lem:bfs}
Let $n>1$ and $k\geq 1$ be integers. Suppose that $H$ is an irreducible $k$-matrix with $n$ columns, 
at least $3$ rows, and a $0$-entry in each row. Then there exists a $k$-matrix $A=A(H)$ with $n$ columns such that\\
\n $(i)$ $A$ has rank $n$ and $m=\frac{1}{2}\left((k+1)^n-k^n-1\right)$ distinct rows,\\
\n $(ii)$ $H=L(A,r_h,\vec{h})$, where $\vec{h}\in\F(A)$ is a rational convex combination of the BFS in $\B(A)$.
\end{lemma}
\begin{proof}
Let $R=(a_1,\ldots,a_n)$ be a row of $H$ and define the complement of $R$ to be the 
vector $R^c=(t-a_1,\ldots,t-a_n)$, where $t=\max_{1\leq i\leq n} a_i$. 
Since $H$ is an irreducible $k$-matrix with at least $3$ rows and 
a $0$-entry in each row, the following conditions hold.\\
\n $(C_1)$ $a_j=0$ for some $j$, $1\leq j\leq n$;\\ 
\n $(C_2)$ $R$ is not the zero row vector;\\ 
\n $(C_3)$ $R^c$ is not a row of $H$.

 Note that $(C_1)$ implies that no row of $H$ is equal to its complement, and 
$(C_3)$ does not hold if $H$ has exactly two rows (namely $R$ and $R^c$).

Let $A=A(H)$ be any $k$-matrix obtained by stacking together one copy of each distinct row of $H$ 
and all the row vectors of length $n$ with entries in $[0,k]=\{0,1,\ldots,k\}$
such that the conditions $(C_1)$--$(C_3)$ hold for every row vector $R$ of $A$. Note that 
there are $(k+1)^n$ row vectors of length $n$ with entries in $[0,k]$, $k^n$ such row vectors 
that do not satisfy $(C_1)$, and one row vector that does not satisfy $(C_2)$. Moreover, for any 
two distinct row vectors $R_i$ and $R_j$ in $A$, the row vectors $R_i,R^c_i,R_j,R^c_j$ are all 
distinct and satisfy $(C_1)$ and $(C_2)$. Thus, the number of rows of $A$ is 
\[m=\frac{1}{2}\left((k+1)^n-k^n-1\right).\]
Furthermore, the matrix $A$ has rank $n$ since the linear span, $\S$, of its rows
contains all the rows of the $n\times n$ identity matrix $I_n$. To see this, first note that 
$\bo^T=(1,\ldots,1)\in \S$ since $H$ is a level matrix and its distinct rows are in $A$. Next, it follows 
from the definition of $A$ that if the row $e_i$ of $I_n$ is not a row of $A$, then its complement 
$e_i^c=(1,\ldots,1)-e_i$ is.  Since $\bo^T$ and $e_i^c$ are in $\S$, then the linear 
combination $\bo^T-e_i^c=e_i$ is also in $\S$. Thus, all the rows of $I_n$ are in $\S$.

By definition of $A$, it follows that $H$ can be obtained (up to a row permutation) 
by stacking some $v_i\geq 0$ copies of $A_i$ (the $i$th row of $A$) for $1\leq i\leq m$. Hence, the vector 
$\vec{v}=(v_1,\ldots,v_m)^T$ satisfies $A^T\vec{v}=t\bo$ for some positive integer $t$. 
If we let  $\vec{h}=\frac{1}{t}\vec{v}$, then 
$H=L(A,t,\vec{h})=L(A,r_h,\vec{h})$, where $r_h=t$ follows from the irreducibility of $H$.
Moreover, $A^T\vec{h}=\bo$ and $h_i\geq 0$ for $1\leq i\leq m$. Thus, $\vec{h}\in \F(A)$.
It follows from the theory of polytopes (e.g., see~\cite[Theorem~$2.3$]{PS}) 
that $\vec{h}$ is a rational convex combination of the BFS in $\B(A)$.
\end{proof}

We define the {\em complement of a matrix} $B$, denoted by $B^c$, to be the matrix obtained by 
replacing each entry $b$ of $B$ by $t-b$, where $t$ is the maximum entry in $B$.

In the next lemma, we give an upper bound on the  number of rows in the matrix $L(A,r_y,\vec{y})$ 
for the case when $A$ is an invertible $k$-matrix and $\vec{y}\in\F(A)$.
\begin{lemma}\label{lem:LG}
Let $n>1$ and $k\geq 1$ be integers.
If $A\in M_{n,n}(\Z)$ is an invertible $k$-matrix, $\vec{y}\in\F(A)$, and $L_A=L(A,r_y,\vec{y})$, then 
\[\ell(L_A)\leq (k/2)^{n-1}(n+1)^{(n+1)/2}.\]
\end{lemma}
\begin{proof}
Since $\vec{y}\in\F(A)$, then $A^T\vec{y}=\bo$. 
Recall that $r_y$ is by definition the smallest positive integer such that
$r_y\vec{y}$ is an integral vector. 
Since $|A^T|\cdot\vec{y}$ is an integral vector by Cramer's rule, it follows that $r_y\leq |A^T|$.

Since $A$ is a $k$-matrix, we may assume that its largest entry is $k$, otherwise $A$ is a $k'$-matrix
with largest entry $k'<k$. 
By definition, $L_A=L(A,r_y,\vec{y})$ is also a $k$-matrix and $L_{A}^c$ denotes its complement.
Since $L_A$ is irreducible by Lemma~\ref{bfs:irr}, then we can directly verify that $L_{A}^c$ is also irreducible. 
Moreover, if we let $A^c$ be the complement of $A$, then there exists $y'\in\F(A^c)$  such that
$L_{A}^c=L_{A^c}=L(A^c,r_{y'},\vec{y'})$,  $r_{y'}\vec{y'}=r_y\vec{y}$, 
 and $\vec{z'}=|(A^c)^T|\cdot ((A^c)^T)^{-1}\bo=|(A^c)^T|\cdot\vec{y'}$ is an integral vector
 (thus, $r_{y'}\leq |(A^c)^T|$).
If $\vec{w}=(a_{1},\ldots,a_{n})$ is the $i$th row vector of $A^T$, then 
$\vec{w^c}=(k-a_{1},\ldots,k-a_{n})$ is the $i$th row vector of $(A^c)^T$.
Since $A^T\vec{y}=\bo$ and $(A^c)^T\vec{y'}=\bo$, it follows that $\vec{w}\cdot\vec{y}=1$ and 
$\vec{w^c}\cdot \vec{y'}=1$. By using these observations and $r_{y'}\vec{y'}=r_y\vec{y}$, we obtain
\[r_y+r_{y'}=\vec{w}\cdot r_y\vec{y}+\vec{w^c}\cdot r_{y'}\vec{y'}
=\sum_{j=1}^n a_{j}r_y y_j+\sum_{j=1}^n (k-a_{j})r_y y_j
=r_y k\sum_{j=1}^n y_j,\]
so that 
\[\sum_{j=1}^n y_i=\frac{r_y+r_{y'}}{r_y k}.\]
Thus, the number of rows of $L_A$ (or $L_{A^c}$) is by definition
\begin{align}\label{eq:det}
\ell(L_A)=r_y\sum_{j=1}^n \vec{y}_j=\frac{r_y+r_{y'}}{k}
&\leq \frac{|A^T|+|(A^c)^T|}{k}\notag\\
&\leq \frac{2(k/2)^n(n+1)^{(n+1)/2}}{k}\\
&= (k/2)^{n-1}(n+1)^{(n+1)/2},\notag
\end{align}
where the inequality~\eqref{eq:det} holds since 
$(k/2)^{n}(n+1)^{(n+1)/2}$ is an upper bound for the determinant of any invertible $n\times n$ matrix 
in which all entries are real and the absolute value of any entry is at most $k$ 
(see~\cite{BC,FS}).
\end{proof}
To prove the main theorem in this section, we also use a theorem of Carath\'{e}odory,
which we shall state after a few definitions, following the account of Ziegler~\cite{GZ}. 

Let $S=\{\vec{x}^{(1)},\ldots, \vec{x}^{(t)}\}$ be a set of vectors from $\R^n$. 
The {\em affine hull} of $S$ is 
\[{\rm Aff}(S)=\{\lambda_1\vec{x}^{(1)}+\ldots+\lambda_t\vec{x}^{(t)} :\;  
\lambda_i\in \R\mbox{ and }\sum_{i=1}^t\lambda_i=1\}.\]
The {\em convex hull} of $S$, which we denote by ${\rm Conv}(S)$, is the set of all its convex
combinations.
A set $I$ of vectors in $\R^n$ is {\em affinely independent} if  every proper subset of $I$ has a 
smaller affine hull.
The dimension of an affine hull $G$ is $g-1$, where $g$ is the cardinality of largest affinely 
independent subset $I\subseteq G$.
Finally, the dimension of a convex hull ${\rm Conv}(S)$ is the dimension of the corresponding 
affine hull ${\rm Aff}(S)$. 
\begin{theorem}[Carath\'{e}odory's Theorem~\cite{GZ}]\label{thm:Ca} 
Let $S=\{\vec{x}^{(1)},\ldots, \vec{x}^{(t)}\}$ be a set of vectors from $\R^n$ such that 
${\rm Conv}(S)$ has dimension $d$. If $\vec{h}\in {\rm Conv}(S)$, then $\vec{h}$ is the 
convex combination of at most $d+1$ properly chosen vectors from $S$.
\end{theorem}
We can now prove the following theorem.
\begin{theorem}\label{thm:UB1}
Let $n>1$ and $k>0$ be integers. If $H$ is an irreducible $k$-matrix with $n$ columns, then
\[\ell(H)\leq k^{n-1}2^{-n}(n+1)^{(n+1)/2}\left((k+1)^n-k^n+1\right).\]
\end{theorem}
\begin{proof}
Let $R=(a_1,\ldots,a_n)$ be a row of $H$ such that $a_0=\min_{1\leq j\leq n}a_j$ is positive.
Then it is easy to verify that the matrix $H'$ obtained from $H$ by replacing 
$R$ with $(a_1-a_0,\ldots,a_n-a_0)$ is also an irreducible $k$-matrix 
with the same number of rows as $H$. Since we are interested in bounding $\ell(n,k)$,
the maximum number of rows of an irreducible $k$-matrix with $n$ columns, we 
may assume (w.l.o.g.) that each row of $H$ contains a $0$-entry.
Since the theorem holds (by inspection) if $H$ has fewer than $3$ rows, we 
may also assume that $H$ has at least $3$ rows.

Thus, it follows from Lemma~\ref{lem:bfs} that there exists a matrix $A=A(H)$ with $m$
rows such that $H=L(A,r_h,\vec{h})$ for some $\vec{h}\in\F(A)$. Moreover, there exist nonnegative 
rational numbers $\lambda_1,\ldots,\lambda_t$, such that 
$\sum_{j=1}^t\lambda_j=1$ and $\vec{h}=\sum_{j=1}^t\lambda_j\vec{x}^{(j)}$, where 
$\vec{x}^{(j)}\in \B(A)$. For $1\leq j\leq t$, recall that $r_j=r_{x^{(j)}}$ 
is the smallest positive integer such that $r_j\vec{x}^{(j)}$ is an integral vector.
Thus,
\begin{eqnarray}\label{thm1:eq1.1}
r_h\vec{h} 
&=&\sum_{j=1}^t \left(\lambda_j\frac{r_h}{r_j} \right) r_j\vec{x}^{(j)}.
\end{eqnarray}
If  $\lambda_j\frac{r_h}{r_j}>1$ for some $j$, then  it follows from~\eqref{thm1:eq1.1} that 
$r_h\vec{h}>r_j\vec{x}^{(j)}$ since $\vec{h}\not=\vec{x}^{(j)}$. This would imply that matrix 
$H'=L(A,r_h,\vec{h})-L(A,r_j,\vec{x}^{(j)})$ (where the subtraction is done componentwise) 
is a proper level $k$-submatrix of $H=L(A,r_h,\vec{h})$, which contradicts the irreducibility of $H$. 
Hence, we must have $\lambda_j\frac{r_h}{r_j}\leq1$ for all $j$. Then this fact and~\eqref{thm1:eq1.1} yield
\begin{eqnarray}\label{thm1:eq1.2}
\ell(H)=r_h\sum_{i=1}^m h_i 
&=& r_h\sum_{j=1}^t \left(\lambda_j\sum_{i=1}^m \vec{x}^{(j)}_i \right)\cr
&=&\sum_{j=1}^t \left(\lambda_j\frac{r_h}{r_j}\sum_{i=1}^m r_j\vec{x}^{(j)}_i \right)\cr
&\leq & \sum_{j=1}^t\sum_{i=1}^m r_j\vec{x}^{(j)}_i.
\end{eqnarray}
Since $\vec{x}^{(j)}\in \B(A)$, there exists a matrix
$L_j=L(A,r_j,\vec{x}^{(j)})$ such that $\ell(L_j)=\sum_{i=1}^m r_j\vec{x}^{(j)}_i$
for $1\leq j\leq t$. Moreover, it follows from Theorem~\ref{thm:Ca} that $t\leq d+1$,
where $d$ is the dimension of the polytope generated by the vectors in $\B(A)$.
Thus, it follows from \eqref{thm1:eq1.2} and the preceding observations that
\begin{equation}\label{eq:mt1}
\ell(H)\leq (d+1)\max_{j} \ell(L_j).
\end{equation}

Since $\vec{x}^{(j)}\in \B(A)$ is a BFS, there exists a subset $I\subseteq \{1,\ldots,m\}$ of $n$ 
indices such that $\vec{x}^{(j)}_i=0$ for $i\not\in I$. Let $A^{(j)}$ be the $n\times n$ matrix 
containing the rows $A_i$ for each $i\in I$. Then $A^{(j)}$ is invertible and $\vec{y}^{(j)}=\left((A^{(j)})^T\right)^{-1}\bo$ satisfies $x^{(j)}_i=x_i$ for all $i\in I$. 
Hence, $L_j=L(A,r_{x^{(j)}},\vec{x}^{(j)})$ and $L_j'=L(A^{(j)},r_{y^{(j)}},\vec{y}^{(j)})$ are 
the same matrices up to a permutation of rows. Thus, it follows from Lemma~\ref{lem:LG} that 
\begin{equation}\label{eq:mt3}
\ell(L_j)=\ell(L_j')\leq (k/2)^{n-1}(n+1)^{(n+1)/2}.
\end{equation}

Since $d$ is at most the number of distinct of rows in $A$, it follows from Lemma~\ref{lem:bfs} that
\begin{equation}\label{eq:mt2}
d+1\leq \frac{(k+1)^n-k^n-1}{2}+1= \frac{(k+1)^n-k^n+1}{2},
\end{equation}
If now follows from~\eqref{eq:mt1}, \eqref{eq:mt2}, and~\eqref{eq:mt3} that 
\[
\ell(H)\leq k^{n-1}2^{-n}(n+1)^{(n+1)/2}\left((k+1)^n-k^n+1\right),
\]
which concludes the proof.
\end{proof}
%
\subsection{The second upper bound for $\ell(n,k)$}\label{sec:UB2}\

In this section, we establish another upper bound for $\ell(n,k)$ that is 
better than the upper bound provided by Theorem~\ref{thm:UB2} for $n\in \{2,3\}$.
\begin{theorem}\label{thm:UB2}
Let $n>1$ and $k>0$ be integers. If $H$ is an irreducible $k$-matrix with $n$ columns,  
then $\ell(H)<(2k)^{2^{n-1} -1}$.
\end{theorem}
\begin{proof}
For convenience define $r_n = {2^{n-1} -1}$. Then $r_2 = 1$ and $r_{n+1}=2r_n+1$.

The proof will be by induction on $n \geq 2$, where the $n=2$ case follows from Theorem~\ref{thm:La}. 
Assume the statement of the theorem holds for $n$, and let $H$ be an irreducible $k$-matrix 
of size $m\times (n+1)$. Let $\widehat{H}$ be the matrix obtained from $H$ by deleting the last column of $H$. 

Note that if a level matrix is reducible, then its rows can be rearranged to form a stack of two level matrices, 
and this division can be continued until a stack of irreducible matrices is attained.
Rearrange the rows of $H$ to form a new matrix 
consisting of a stack of $k$-matrices 
$M^{(i)}$ of size $m_i\times(n+1)$, $1 \leq i \leq t$, such that each matrix $\widehat{M}^{(i)}$ 
formed by deleting the last column of $M^{(i)}$ is an irreducible matrix. 
For $1 \leq i \leq t$, let $a_i$ be the common column sum of $\widehat{M}^{(i)}$, and let $b_i$ be 
the sum of the entries in the last column of $M^{(i)}$.  
Note that for $1 \leq i \leq t$, we have $m_i \leq  (2k)^{r_n}$ by the induction hypothesis.
Since $H$ is a $k$-matrix, both $a_i$ and $b_i$ cannot exceed $K=(2k)^{r_n}k$. 

Let $M$ be the $K$-matrix of size $t\times 2$ with $i$th row vector $(a_i,b_i)$, $1 \leq i \leq t$. 
Since $H$ is level, so is $M$. Since $m_i \leq  (2k)^{r_n}$, we also have
\[m=\sum_{i=1}^t m_i \leq \sum_{i=1}^t(2k)^{r_n} = t(2k)^{r_n}.\]
Since $H$ is irreducible, then $M$ is also irreducible, and $t\leq 2K-1$ by Theorem~\ref{thm:La}.  
Thus 
\[m \leq t(2k)^{r_n} <2K(2k)^{r_n}= 2k((2k)^{r_n})^2 = (2k)^{2r_n+1}= (2k)^{r_{n+1}}.\]
\end{proof}
\subsection{The lower bound}\label{sec:LB}\

In this section we construct examples of irreducible matrices, each of whose number of rows is 
greater than some exponential function of its number of columns.
\begin{theorem}\label{thm:LB} 
Let $\epsilon>0$ be any real number.
There exist infinitely many irreducible $1$-matrices with 
$n$ columns and more than $e^{(1-\epsilon)\sqrt{n\ln n}}$ rows. 
\end{theorem}
\begin{proof}
We write $f(x) \sim g(x)$ if $\lim_{x \to \infty} f(x)/g(x) = 1$.

Let $p_i$ denote the $i$th prime,  and for $x \geq 2$, let $t=\pi(x)$ be the number of primes not 
exceeding $x$. Let $J_r$ be the $r\times r$ matrix 
with $0$'s on its main diagonal and $1$'s everywhere else. 
Set $r_i= p_i + 1$, $n=n(x)=\sum_{i=1}^t r_i$, and consider the $n \times n$ matrix 
\begin{equation*} 
A =A(x)= 
\left[ 
\begin{matrix}
J_{r_1} & 0 & 0 & \cdots & 0\\
0 & J_{r_2} & 0 & \cdots & 0\\
& & \cdots & &\\
0 & 0 &   \cdots & 0 & J_{r_t} 
\end{matrix}
\right],
 \end{equation*}
where the 0's represent zero matrices of the appropriate sizes. 
By repeating each row of $J_{r_i}$ in $A$ exactly $P/p_i$ times, where 
$P=P(x)= \prod_{i=1}^t p_i$, we get a level matrix $A^*=A^*(x)$. 
Since $p_1, p_2, \ldots,p_t$ are relatively prime in pairs, 
the matrix $A^*$ is irreducible, and the number of columns of $A^*$ is still $n$. 

Let $\theta=\theta(x)=\ln \left(\prod_{i=1}^t p_i\right)$.
The prime number theorem states that $\pi(x) \sim x/ \ln x$, and this is equivalent to $\theta\sim x$
(see~\cite{Ap}, Th. 4.4).  It also follows from~\cite{La} that $\sum_{i=1}^tp_i\sim x^2/\ln(x^2)$, 
and from~\cite{Ap}, Th. 12, that $\sum_{i=1}^t\frac{1}{p_i} \sim \ln \ln x$.

Now the number of rows of $A^*$ is 
\[m=m(x)=\sum_{i=1}^t\frac{P}{p_i}(p_i+1).\]
Thus
\[m= P\sum_{i=1}^t(1 +1/p_i) \mbox{ and } \ln m=\theta(x)+ \ln\left(\sum_{i=1}^t (1 +1/p_i)\right).\]
These relations yield $\ln m\sim x$ and $n \sim x^2/\ln (x^2)$.

Notice that $f(y) = \sqrt{y \ln y}$ is increasing for $y \geq 1$.
Let $y$ be such that $x=\sqrt{y\ln y}$, and note that $x \to \infty$ if and only if $y \to \infty$.  
Then $\ln m\sim x$ and $n \sim x^2/\ln (x^2)$ imply that 
\[\ln m\sim \sqrt{y\ln y} \mbox{ and } 
n \sim \frac{x^2}{\ln x^2}=\frac{y\ln y}{\ln (y\ln y)}= \frac{y\ln y}{\ln y+\ln\ln y}\sim y.\]
Hence $\ln m\sim \sqrt{n\ln n}$. Thus, for any $\epsilon>0$, there exists an integer $n=n(x,\epsilon)$ 
such that $\ln m / \sqrt{n\ln n} > 1 - \epsilon$,  or 
\[ m> e^{(1-\epsilon)\sqrt{n\ln n}}.\]
\end{proof}
\subsection{Proof of Theorem~\ref{thm:main}}\label{sec:proof}\

The first part Theorem~\ref{thm:main} follows from directly from Theorem~\ref{thm:LB}.
The upper bound for $\ell(n,k)$ follows from Theorem~\ref{thm:UB1} when $n=3$, and from 
Theorem~\ref{thm:UB2} when $n>3$.
\section{Application to  multipartitions of finite vector spaces}\label{sec:app-vsp}
Let $V=V(n,q)$, where $V(n,q)$ denotes the $n$-dimensional vector space over the finite field 
with $q$ elements. We will consider multisets of nonzero subspaces of $V$ such 
that  each nonzero element of $V$ is in the same number of subspaces, counting multiplicities. 
More explicitly, a {\em multipartition} $P$ of $V$ is a pair $(F, \alpha)$, where $F$ is a finite set, 
$\alpha$ is a function from $F$ to the set of nonzero subspaces of $V$, and there exists a 
positive integer $\lambda$ such that whenever  $v$ is a nonzero elements of $V$ we have 
\[ |\{f \in F :\;  v \in \alpha(f)\}|=\lambda.\]
In this case,  we call $P$ a $\lambda$-{\em partition}.

A number of papers have been written about $1$-partitions, usually just called ``partitions'',
(e.g., see~\cite{Be,Bu,ESSSV1,He1,Li} and \cite{He2} for a survey), and at least one about 
multipartitions (\cite{ESSSV2}). 
A general question in this area is to classify the multipartitions of $V$.

If $V$ has a $\lambda$-partition $P$ and a $\mu$-partition $Q$, then  a 
$(\lambda + \mu)$-partition of $V$ may be formed by combining $P$ and $Q$ in the obvious way. 
We denote this by $P + Q$. 
Conversely, it may be possible to break a multipartition into smaller multipartitions.  
Thus, it is of interest to investigate multipartitions that cannot be broken up any further.  
We call a multipartition $P$  of $V$ {\em irreducible} if there do not exist multipartitions 
$Q_1$ and $Q_2$ of $V$ such that $P = Q_1 + Q_2$. 
Clearly any multipartition of $V$ can be written as a sum of irreducible multipartitions of $V$.

Let $S$ be a set.  Call $(F,\alpha)$ a {\em level family} of $S$ if $\alpha$ is a function from $F$ into 
$2^S\backslash \{\emptyset \}$ for which there exists a positive integer $\lambda$ 
such that if $x \in S$, then $|\{f \in F :\;  \alpha(f)=x\}| = \lambda$. We call $\lambda$ 
 the {\em height} of the family. 
Call the level family $(F, \alpha)$ with height $\lambda$ {\em reducible} if there exists a subset $F'$ of $F$ 
and an integer $\lambda'$, 
$0 < \lambda' < \lambda$, such that $(F', \alpha|_{F'})$ is a level family of $S$ with height $\lambda'$; 
otherwise call $(F, \alpha)$ {\em irreducible}. 

\begin{corollary}
If $S$ has $n \geq 2$ elements and $(F,\alpha)$ is an irreducible family of $S$, then
\[|F|\leq (n+1)^{(n+1)/2}.\]
Thus, a finite set $S$ has only finitely-many irreducible families.
\end{corollary}
\begin{proof}
Let $s_1,\ldots,s_n$ be the distinct elements of $S$, and let $f_1,\ldots,f_m$ be the distinct 
elements of $F$.
Define the $1$-matrix $B$ with entries $b_{1,1},\ldots,b_{m,n}$ by letting $b_{ij} = 1 $ if 
$s_j \in \alpha(f_i)$, and $b_{ij}=0$ otherwise.
By the definition of an irreducible family all the column sums of $B$ are equal, and $B$ is an 
irreducible matrix. Then by setting $k=1$ in Theorem~\ref{thm:main}, we obtain $|F| = m\leq (n+1)^{(n+1)/2}$.
\end{proof}
\begin{corollary}
If $(F,\alpha)$ be an irreducible $\lambda$-partition of $V(n,q)$ for some integer $\lambda$, then 
\[ |F|\leq q^{(n-1)q^{n-1}/2}.\]  
Thus, $V(n,q)$ has finitely--many irreducible $\lambda$-partitions.
\end{corollary}
\begin{proof}
Let $W_1,\ldots,W_t$ be the distinct 1-dimensional subspaces of $V(n,q)$, where $t=(q^n-1)/(q-1)$, 
and let $f_1,\ldots, f_m$ be the distinct elements of $F$. 
Define the $1$-matrix $B$ with entries $b_{1,1},\ldots,b_{m,t}$ by letting  $b_{ij} = 1$ if 
$W_j\subseteq\alpha(f_i)$ and $b_{ij}=0$ otherwise.
Then all the column sums of $B$ are $\lambda$, and $B$ is an irreducible matrix. 
Since $t+1=q^{n-1}$, setting $k=1$ in Theorem~\ref{thm:main} yields
\[ |F|=m\leq (t+1)^{(t+1)/2}=q^{(n-1)q^{n-1}/2}.\] 
\end{proof}
\section{Conclusion}\label{sec:conc}
Our main question (see page ~1) is still open in general.
For example, if $A\in M_{m,n}(\Z)$ is a $k$-matrix such that $A^T\vec{x}=\bo$
for some $\vec{x}>\vec{0}$, then we know by Lemma~\ref{lem:LG} that the number of rows, $\ell(L_A)$, 
of the irreducible matrix $L_A=L(A,r_x,\vec{x})$ satisfies $\ell(L_A)\leq (k/2)^{n-1}(n+1)^{(n+1)/2}$. 
However, the small cases that we have checked suggest that $\ell(L_A)$ is much smaller.
It would be interesting to find the exact value of $\ell(L_A)$ or improve its upper bound.
Such an improvement would also give a better upper bound for the general 
value of $\ell(n,k)$ in Theorem~\ref{thm:main}.

In Proposition~\ref{prop:hil}, we characterized the set $\Zc(A)$ of levelers of a given matrix by a cone
whose Hilbert basis is the set of irreducible levelers $\I(A)\subseteq \Zc(A)$. It would be interesting
to investigate if this characterization can shed more light on the study of Hilbert bases 
(e.g., see~\cite{CFS,HW}) in certain cases. 

\bs\n{\bf Acknowledgement:}\
The authors thank S. Tipnis for suggesting the connection to polyhedral cones.

\end{document}